\documentclass{amsart}
\usepackage{amsthm,amssymb,mathtools}
\usepackage{stmaryrd,tikz,pgfplots}
\usepackage{marginnote}

\usepackage[textsize=tiny]{todonotes}

\newcommand{\Marginpar}[1]{\marginpar{\tiny{#1}}}
\newcommand{\Note}[1]{{\par\noindent\hrulefill\par\tiny{#1}\par\noindent\hrulefill\par}}
\newcommand{\Detail}[1]{{#1}}
\renewcommand{\Marginpar}[1]{}
\renewcommand{\Note}[1]{}
\renewcommand{\Detail}[1]{}

\usepackage{hyperref}
\usepackage[all]{xy}
\usepackage{amsmath,amsfonts}	
\usepackage{amsthm,latexsym,amssymb}
\usepackage{tensor}

\usepackage{booktabs} 
\usepackage{siunitx}

\newtheorem{thm}{Theorem}
\newtheorem{prop}[thm]{Proposition}

\newtheorem{cor}[thm]{Corollary}

\theoremstyle{definition}
\newtheorem{defn}[thm]{Definition}

\newtheorem{rem}[thm]{Remark}

\renewcommand{\[}{\begin{equation*}}
\renewcommand{\]}{\end{equation*}}

\def\z{\mathfrak{z}}

\DeclareMathOperator\vol{vol}

\begin{document}
\parskip1mm

\title[Hermitian metrics of constant Chern scalar curvature]{Hermitian metrics of constant Chern scalar curvature on ruled surfaces}

\author{Caner Koca}
\address{Department of Mathematics, NYC College of Technology of CUNY, Brooklyn, NY 11021, USA.}
\email{ckoca@citytech.cuny.edu}

\author{Mehdi Lejmi}
\address{Department of Mathematics, Bronx Community College of CUNY, Bronx, NY 10453, USA.}
\email{mehdi.lejmi@bcc.cuny.edu}

\thanks{The first author was supported in part by the PSC-CUNY award \#61768-00 49, jointly funded by The Professional Staff Congress and The City University of New York.
The second author is supported by the Simons Foundation Grant \#636075. The second author was also supported in Summer 2019 by the PSC-CUNY award\# 62380-00 50, jointly funded by The Professional Staff Congress and The City University of New York.}

\begin{abstract}
It is known that Hirzebruch surfaces of non zero degree do not admit any constant scalar curvature K\"ahler metric~\cite{MR2807093,gauduchonbook,Martinez-Garcia:2017aa}.
In this note, we describe how to construct Hermitian metrics of positive constant Chern scalar curvature on Hirzebruch surfaces using Page--B\'erard-Bergery's ansatz~\cite{page1978compact,MR727843}.
We also construct the interesting case of Hermitian metrics of zero Chern scalar curvature on some ruled surfaces.
Furthermore, we discuss the problem of the existence in a conformal class of critical metrics of the total Chern scalar curvature, studied by Gauduchon in~\cite{MR567760,MR742896}.
\end{abstract}
\maketitle
\section{introduction}
A Hermitian manifold $(M,J,g)$ of real dimension $2n$ is a manifold $M$ equipped with an integrable almost-complex structure $J$ and a Riemannian metric
such that $J$ is orthogonal with respect to $g$. The metric $g$ is said to be K\"ahler if the induced fundamental form $F(\cdot,\cdot):=g(J\cdot,\cdot)$ is closed.
On a Hermitian manifold, one can consider the Chern connection $\nabla$ defined as the unique Hermitian connection
with $J$-anti-invariant torsion~\cite{MR0066733,MR0165458,MR1456265}. The Chern connection actually induces the canonical Cauchy--Riemann operator
on the Hermitian tangent bundle $(T(M),J,g)$ viewed as a holomorphic bundle. From the Chern connection, one can derive the Chern scalar curvature $s^C$ (see Definition~\ref{def_scalar}). On a closed Hermitian manifold $(M,J,g)$, the Chern scalar curvature does not coincide with the Riemannian scalar curvature (derived from the Levi-Civita connection)
unless the metric $g$ is K\"ahler~\cite{MR742896,MR3632564}.

In the general (almost-)Hermitian setting, it is natural to study the existence of metrics of constant Chern scalar curvature
(see for example \cite{MR2795448,MR2917134,MR2988734,MR3696598,Angella:aa,Shen:aa} and the references therein).
For instance, Angella, Calamai and Spotti~\cite{MR3696598} initiated the Chern--Yamabe problem (see also~\cite{MR0301680,MR779217}), that is an analogue of the Yamabe problem~\cite{MR0125546}, namely they studied the existence and the uniqueness of constant Chern scalar curvature metrics in the conformal class of $g$ on a closed Hermitian manifold $(M,J,g)$. In particular, they proved that when the fundamental constant $C(J,[g])$ (see Definition~\ref{fundamental_constant} or for instance~\cite{MR0486672,MR779217,MR1712115}) is negative or zero
then there exists in the conformal class a metric with the Chern scalar curvature equal to $C(J,[g])$. Later, the problem was extended to the almost-Hermitian setting
in~\cite{Lejmi:2017aa}.  
Moreover, a parabolic version of the Chern--Yamabe problem was studied in~\cite{MR3833814,Calamai:aa,ho2019results}. However,
the problem remains widely open when the fundamental constant $C(J,[g])$ is positive. We remark that
for complex surfaces $C(J,[g])$ being positive implies that the complex surface has to be of negative Kodaira dimension by vanishing theorems of Gauduchon~\cite{MR0470920}.

The $m$-Hirzebruch surfaces $\mathbb{M}_m=\mathbb{P}\left(\mathcal{O}_{\mathbb{CP}^1}\oplus \mathcal{O}_{\mathbb{CP}^1}(m)  \right)\longrightarrow\mathbb{CP}^1$ are of negative Kodaira dimension and K\"ahler
(here $\mathbb{CP}^1$ is the complex projective space of complex dimension $1$ and $m\geq 0$ is an integer),
however it is known that $\mathbb{M}_m$ {does not admit any constant scalar curvature K\"ahler metric} when $m>0$~\cite{MR2807093,gauduchonbook,Martinez-Garcia:2017aa}. 

After the preliminaries in Section~\ref{Sec_2}, we describe in Section~\ref{Sec_3} how to construct on the Hirzebruch surfaces $\mathbb{M}_m$ (with $m>0$) Hermitian metrics of positive constant Chern scalar curvature  
using a well-known ansatz due to Page~\cite{page1978compact} and generalized by B\'erard-Bergery~\cite{MR727843} by considering the $U(2)$-invariant metrics of the form
\begin{equation}\label{metric_intro}
g=h(t)^2\left(e^1\otimes e^1+e^2\otimes e^2 \right)+f(t)^2e^3\otimes e^3+dt^2,
\end{equation}
where $t$ is a coordinate transverse to the $U(2)$-orbits, and $e^1,e^2,e^3$ are the invariant $1$-forms on $S^3$ dual to the vectors $X,Y,V$ respectively,
and $h(t)$ and $f(t)$ are positive functions satisfying certain boundary conditions. We obtain then the following
\begin{thm}\label{chern_constant}
There exists on the $m$-Hirzebruch surface (with $m>0$) conformally K\"ahler metrics of positive constant Chern scalar curvature of the form~(\ref{metric_intro}).
\end{thm}
From the Chern connection, one can also derive the third scalar curvature $s$ (see Definition~\ref{def_scalar} or for instance~\cite{MR742896}). On a closed Hermitian manifold of real dimension $2n=4$, $s$ and $s^C$ coincide if and only the metric is K\"ahler~\cite{MR742896}. Also note that the conformal changes of $s^C$ and $s$ are similar~\cite{MR742896}.
It is then interesting to construct Hermitian metrics of positive constant third scalar curvature. We succeed to construct conformally K\"ahler metrics of positive constant third scalar scalar curvature of the form~(\ref{metric_intro}) on the $1$-Hirzebruch surface (see Theorem~\ref{third_constant}). 
In Section~{\ref{Sec_4}}, we use the generalized Calabi construction~\cite{MR2144249,MR2425136,MR2807093} on minimal ruled surfaces  to construct metrics of zero Chern scalar curvature on some of these surfaces
(see Proposition~\ref{Prop_zero_Chern}). Finally, in Section~\ref{Sec_5}, we discuss the existence in a conformal class of the critical metrics of the total Chern scalar curvature studied by Gauduchon in~\cite{MR567760,MR742896}.

\section{Preliminaries}\label{Sec_2}
Let $(M,J,g)$ be an almost-Hermitian manifold of real dimension $2n.$ Hence, $J:T(M)\longrightarrow T(M)$ is an almost-complex structure 
i.e. $J^2=-\mathrm{Id}|_{T(M)}$ and $g$ is a Riemannian metric compatible with the almost-complex structure $J$ i.e. $g(J\cdot,J\cdot)=g(\cdot,\cdot).$
The pair $(J,g)$ induces a $2$-form $F(\cdot,\cdot):=g(J\cdot,\cdot)$. The $2$-form $F$ is called the fundamental form and it is not necessarily closed. In fact,
$$dF=\left(dF\right)_0+\frac{1}{n-1}\theta\wedge F,$$
where $d$ is the exterior derivative, $\left(dF\right)_0$ is the primitive part and $\theta$ is $1$-form called the Lee form.
The metric $g$ is called Gauduchon if $\delta^g\theta=0,$ where $\delta^g$ is the codifferential defined as the adjoint of $d$. It turns out that
the conformal class of any almost-Hermitian metric $g$ contains a unique (up to a constant) Gauduchon metric~\cite{MR0470920}.
Now, the almost-Hermitian structure $(J,g)$ is called Hermitian
if $J$ is integrable that is equivalent to the vanishing of the Nijenhuis tensor $N$~\cite{MR0088770} defined as
$$4N(\cdot,\cdot)=[J\cdot,J\cdot]-[\cdot,\cdot]-J[J\cdot,\cdot]-J[\cdot,J\cdot].$$
A Hermitian structure $(J,g)$ is K\"ahler if the induced fundamental form if closed i.e. $dF=0.$

On an almost-Hermitian manifold $(M,J,g)$, the almost-complex $J$ is parallel with respect to the Levi-Civita connection $D^g$
if and only if $g$ is K\"ahler. It is natural then to consider the Chern connection $\nabla$~\cite{MR0066733,MR0165458,MR1456265} defined as the unique connection
satisfying $\nabla J=\nabla g=0$ and $T^\nabla_{JX,JY}=-T^\nabla_{X,Y},$ where $T^\nabla$ is the torsion of $\nabla$ defined
by $T^\nabla_{X,Y}=\nabla_XY-\nabla_YX-[X,Y],$ for any vector fields $X,Y.$ Moreover, The $(0,1)$-part of the Chern connection, with
respect to $J,$ corresponds to the canonical Cauchy--Riemann operator defined on the tangent bundle $T(M)$~\cite{MR1456265}.

Let $R^\nabla_{X,Y}=[\nabla_X,\nabla_Y]-\nabla_{[X,Y]}$ be the curvature of the Chern connection $\nabla.$
Then, the real $2$-form $\rho^\nabla(X,Y)=-\Lambda g\left(R^\nabla_{X,Y}\cdot,\cdot\right)$ is called the first (or Hermitian) Ricci form. Here, $\Lambda$
stands for the contraction by the fundamental form $F$ induced by $(J,g)$. The first Ricci form
$\rho^\nabla$ is a representative of the first Chern class $2\pi c_1\left(T(M),J\right),$ in particular $d\rho^\nabla=0$.
If $J$ is integrable then $\rho^\nabla$ is $J$-invariant. On the other hand, one can also define the second Ricci form $r$ by $r(X,Y)=-\Lambda g\left(R_{\cdot,\cdot}\, X,Y\right)$ (see for instance~\cite{MR742896}). 
\begin{defn}\label{def_scalar}
The Chern scalar curvature $s^C$ and the third scalar curvature $s$ are defined by
$$s^C=\Lambda\left(\rho^\nabla\right)=\Lambda\left(r\right)=-\frac{1}{4}\sum_{i,j=1}^nR^\nabla(e_i,Je_i,e_j,Je_j).$$
$$s=-\frac{1}{2}\sum_{i,j=1}^nR^\nabla(e_i,e_j,e_i,e_j),$$
where $\{e_1,\cdots,e_{2n}\}$ is a local $g$-orthonormal basis of $T(M)$ such that $e_{i+n}=Je_{i}.$
\end{defn}
On a closed almost-Hermitian manifold, the Chern scalar curvature $s^C$, the third scalar curvature $s$
and the Riemannian scalar curvature $s^g$ (with respect to the Levi-Civita connection $D^g$) do not coincide in general (see~\cite{MR742896,MR3632564,Lejmi:2017aa}). However, if $g$ is K\"ahler then $s^C=s=\frac{1}{2}s^g.$

\section{The Chern and the third scalar curvatures on the $m$-Hirzebruch surfaces}\label{Sec_3}

We consider the $U(2)$-invariant metric
\begin{equation}\label{metric}
g=h(t)^2\left(e^1\otimes e^1+e^2\otimes e^2 \right)+f(t)^2e^3\otimes e^3+dt^2,
\end{equation}
where $t$ is a coordinate transverse to the $U(2)$-orbits, and $e^1,e^2,e^3$ are the invariant $1$-forms on $S^3$ dual to the vectors $X,Y,V$ respectively
satisfying $[X,Y]=2V,[Y,V]=2X,[V,X]=2Y.$ It is well-known that Page~\cite{page1978compact} constructed a non-homogeneous Einstein metric of the form~(\ref{metric}) on $\mathbb{CP}^2\#\overline{\mathbb{CP}}^2$. The construction was then 
generalized by B\'erard-Bergery~\cite{MR727843} and used to provide for instance compact examples of Einstein-Weyl structures in~\cite{swann1993einstein} or $\ast$-Einstein metrics in~\cite{MR1723831} (see also~\cite{MR1246192})
or constant Riemannian scalar curvature metrics in~\cite{MR3263197}. 

On the $m$-Hirzebruch surface, $h(t)$ and $f(t)$ are positive functions on the interval $(0,l)$ and satisfy the following boundary conditions:
\begin{equation}\label{boundary_conditions}
f'(0)=-f'(l)=m,f^{(2p)}( 0)=f^{(2p)}( l)=0,\quad\forall p\geq 0.
\end{equation}
\begin{equation}\label{boundary_conditions_2}
h(0)>0, h(l)>0,h^{(2p+1)}(0)=h^{(2p+1)}(l)=0,\quad\forall p\geq 0.
\end{equation}
In particular, on $\mathbb{CP}^2\# \overline{\mathbb{CP}}^2$, we have $m=1.$

Let us consider the $g$-orthonormal basis $\{E_1=\frac{1}{h}X,E_2=\frac{1}{h}Y,E_3=\frac{1}{f}V,E_4=\frac{\partial}{\partial t}\}$ of the tangent bundle
and its dual $\{\sigma_1=h{e^1},\sigma_2=h{e^2},\sigma_3=f{e^3},\sigma_4=dt\}$ as a basis of the cotangent bundle.
A direct computation shows that:
$$[E_1,E_2]=\frac{2f}{h^2}E_3,\quad [E_1,E_3]=-\frac{2}{f}E_2,\quad [E_2,E_3]=\frac{2}{f}E_1,$$
$$ [E_1,E_4]=\frac{h'}{h}E_1,\quad[E_2,E_4]=\frac{h'}{h}E_2,\quad  \quad[E_3,E_4]=\frac{f'}{f}E_3.$$
We consider then the integrable almost-complex $J$ given by
\begin{equation}\label{complex_structure}
JE_1=E_2,\,\,JE_3=E_4.
\end{equation}
Then $J$ is compatible with the metric $g$ inducing the fundamental form
\begin{equation*}
F=\sigma_1\wedge \sigma_2+\sigma_3\wedge\sigma_4.
\end{equation*}
We compute then
\begin{eqnarray*}
dF&=&d(h^2\,e^1\wedge e^2+f\,e^3\wedge dt),\\
&=&2hh' \,dt\wedge e^1\wedge e^2+2f \, e^2\wedge e^1\wedge dt,\\
&=&\frac{2(hh'-f)}{h^2}\,dt\wedge\omega.
\end{eqnarray*}
The metric $g$ is then locally conformally K\"ahler. Indeed, the Lee form $\theta$ of $(J,g)$ is closed and it is given by
\begin{equation}\label{lee_form}
\theta=\frac{2(hh'-f)}{h^2}\,dt.
\end{equation}
In particular, $g$ is K\"ahler if and only if
\begin{equation}\label{K_condition}
f=hh'.
\end{equation}

\subsection{Chern curvature computations}

We denote by $\nabla$ the Chern connection of the Hermitian structure $(J,g).$ We would like to compute the curvature $R^\nabla$ of the Chern connection.
We recall the following formula~(\cite[Proposition 9.3.2]{gauduchonbook}) of the Chern connection $\nabla$: for any vector fields $X,Y,Z$, we have
\begin{eqnarray*}
g(\nabla_XZ,Y)&=&\frac{1}{2}X.\,g(Z,Y)+\frac{1}{2}JX.\,g(Z,JY)\\
&+&\frac{1}{4}g\left([X,Z]+[JX,JZ]+J[JX,Z]-J[X,JZ],Y \right)\\
&-&\frac{1}{4}g\left([X,Y]+[JX,JY]+J[JX,Y]-J[X,JY],Z \right)
\end{eqnarray*}
For the Hermitian structure $(J,g)$ given by~(\ref{complex_structure}) and~(\ref{metric}), we get
\medskip
{\large
\begin{center}
\renewcommand{\arraystretch}{2}{
\begin{tabular}{ |c|c|c|c|c| } 
\hline
$\nabla$ & $E_1$ & $E_2$ & $E_3$ & $E_4$ \\
\hline
$E_1$ & $\frac{-f}{h^2}E_4$ & $\frac{f}{h^2}E_3$   & $\frac{-f}{h^2}E_2$ & $\frac{f}{h^2}E_1$  \\
\hline
$E_2$ & $\frac{-f}{h^2}E_3$ & $\frac{-f}{h^2}E_4$ &  $\frac{f}{h^2}E_1$ & $\frac{f}{h^2}E_2$  \\
\hline
$E_3$ & $\frac{-fh'+2h}{fh}E_2$ & $\frac{fh'-2h}{fh}E_1$  & $\frac{-f'}{f}E_4$ & $\frac{f'}{f}E_3$ \\
\hline
$E_4$ & $0$ & $0$ & $0$  & $0$ \\
\hline
\end{tabular}}
\end{center}
}
\medskip

For instance $\nabla_{E_1}E_2=\frac{f}{h^2}E_3,$ etc. Now, for the second derivatives, we obtain
\medskip
{\large
\begin{center}
\renewcommand{\arraystretch}{2}{
\begin{tabular}{ |c|c|c|c|c| } 
\hline
$\nabla_{\bullet}\nabla_{\bullet} E_1$ & $E_1$ & $E_2$ & $E_3$ & $E_4$ \\
\hline
$E_1$ & $\frac{-f^2}{h^4}E_1$ & $\frac{f^2}{h^4}E_2$   &  $\frac{-fh'+2h}{h^3}E_3$ & $0$  \\
\hline
$E_2$ & $\frac{-f^2}{h^4}E_2$ & $\frac{-f^2}{h^4}E_1$ & $\frac{fh'-2h}{h^3}E_4$  &   $0$\\
\hline
$E_3$ & $\frac{-f'}{h^2}E_3$ & $\frac{f'}{h^2}E_4$  &  $-\frac{(fh'-2h)^2}{h^2f^2}E_1$   & $0$  \\
\hline
$E_4$ & $\left(\frac{-f}{h^2}\right)^{'}E_4$ & $\left(\frac{-f}{h^2}\right)^{'}E_3$ & $\left(\frac{-fh'+2h}{fh}\right)^{'}E_2$   & $0$ \\
\hline
\end{tabular}}
\end{center}
}
\medskip

{\large
\begin{center}
\renewcommand{\arraystretch}{2}{
\begin{tabular}{ |c|c|c|c|c| } 
\hline
$\nabla_{\bullet}\nabla_{\bullet} E_3$ & $E_1$ & $E_2$ & $E_3$ & $E_4$ \\
\hline
$E_1$   &  $\frac{-f^2}{h^4}E_3$  &  $\frac{-f^2}{h^4}E_4$  & $\frac{-f'}{h^2}E_1$   & $0$\\
\hline
$E_2$ & $\frac{f^2}{h^4}E_4$ & $\frac{-f^2}{h^4}E_3$& $\frac{-f'}{h^2}E_2$  & $0$\\
\hline
$E_3$ &  $\frac{-fh'+2h}{h^3}E_1$& $\frac{-fh'+2h}{h^3}E_2$ &$\frac{-f'^2}{f^2}E_3$  &  $0$\\
\hline
$E_4$ & $\left(\frac{-f}{h^2}\right)^{'} E_2$ & $\left(\frac{f}{h^2}\right)^{'} E_1$ & $-\left(\frac{f'}{f}\right)^{'} E_4$  & $0$ \\
\hline
\end{tabular}}
\end{center}
}
\medskip

For example $\nabla_{E_1}\nabla_{E_2} E_1=\frac{f^2}{h^4}E_2,$ etc. For the derivative with respect to the brackets, we have
\medskip
{\large
\begin{center}
\renewcommand{\arraystretch}{2}{
\begin{tabular}{ |c|c|c|c|c| } 
\hline
$\nabla_{[\bullet,\bullet]}E_1$ & $E_1$ & $E_2$ & $E_3$ & $E_4$ \\
\hline
$E_1$   & $0$  & $\frac{-2fh'+4h}{h^3}E_2$   & $\frac{2}{h^2}E_3$   & $\frac{-fh'}{h^3}E_4$ \\
\hline
$E_2$ &  $\frac{2fh'-4h}{h^3}E_2$ &$0$ & $\frac{-2}{h^2}E_4$  & $\frac{-fh'}{h^3}E_3$ \\
\hline
$E_3$ & $-\frac{2}{h^2}E_3$ & $\frac{2}{h^2}E_4$   & $0$&  $\frac{-f'(fh'-2h)}{hf^2}E_2$  \\
\hline
$E_4$ & $\frac{fh'}{h^3}E_4$& $\frac{fh'}{h^3}E_3$   &  $\frac{f'(fh'-2h)}{hf^2}E_2$ & $0$ \\
\hline
\end{tabular}}
\end{center}
}
\medskip
{\large
\begin{center}
\renewcommand{\arraystretch}{2}{
\begin{tabular}{ |c|c|c|c|c| } 
\hline
$\nabla_{[\bullet,\bullet]}E_3$ & $E_1$ & $E_2$ & $E_3$ & $E_4$ \\
\hline
$E_1$   & $0$  &  $\frac{-2f'}{h^2}E_4$ &  $\frac{-2}{h^2}E_1$  & $\frac{-fh'}{h^3}E_2$  \\
\hline
$E_2$ & $\frac{2f'}{h^2}E_4$ &$0$ & $\frac{-2}{h^2}E_2$  & $\frac{fh'}{h^3}E_1$ \\
\hline
$E_3$ & $\frac{2}{h^2}E_1$  & $\frac{2}{h^2}E_2$  & $0$& $\frac{2-f'^2}{f^2}E_4$    \\
\hline
$E_4$ &  $\frac{fh'}{h^3}E_2$ & $\frac{-fh'}{h^3}E_1$   & $\frac{-2+f'^2}{f^2}E_4$ & $0$ \\
\hline
\end{tabular}}
\end{center}
}
\medskip

For example $\nabla_{[E_1,E_2]}E_1=\frac{-2fh'+4h}{h^3}E_2,$ etc. Now, we compute the components of the curvature $R^\nabla$
\medskip
\newline
{\large
\renewcommand{\arraystretch}{2}{\hspace*{-6em}
\begin{tabular}{ |c|c|c|c|c| } 
\hline
$R^\nabla_{\bullet,\bullet}E_1$ & $E_1$ & $E_2$ & $E_3$ & $E_4$ \\
\hline
$E_1$   & $0$  & $\frac{2h'fh+2f^2-4h^2}{h^4}E_2$ & $\frac{f'h-fh'}{h^3}E_3$   & $\frac{f'h-h'f}{h^3}E_4$   \\
\hline
$E_2$ & $\frac{-2h'fh-2f^2+4h^2}{h^4}E_2$ & $0$ &  $\frac{-f'h+fh'}{h^3}E_4$ &$\frac{f'h-h'f}{h^3}E_3$  \\
\hline
$E_3$ & $\frac{-f'h+fh'}{h^3}E_3$ &  $\frac{f'h-fh'}{h^3}E_4$ & $0$&  $\left(\left( \frac{fh'-2h}{hf} \right)^{'}+\frac{f'(fh'-2h)}{f^2h}\right)E_2$   \\
\hline
$E_4$ &  $-\frac{f'h-h'f}{h^3}E_4$&  $-\frac{f'h-h'f}{h^3}E_3$ & $-\left(\left( \frac{fh'-2h}{hf} \right)^{'}+\frac{f'(fh'-2h)}{f^2h}\right)E_2$ & $0$ \\
\hline
\end{tabular}}
}
\medskip

%
%
%
%
%

{\large
\begin{center}
\renewcommand{\arraystretch}{2}{
\begin{tabular}{ |c|c|c|c|c| } 
\hline
$R^\nabla_{\bullet,\bullet}E_3$ & $E_1$ & $E_2$ & $E_3$ & $E_4$ \\
\hline
$E_1$   & $0$  & $\left(-\frac{2f^2}{h^4}+\frac{2f'}{h^2}\right)E_4$  &  $\frac{-f'h+fh'}{h^3}E_1$  &  $\frac{f'h-h'f}{h^3}E_2$  \\
\hline
$E_2$ & $\left(\frac{2f^2}{h^4}-\frac{2f'}{h^2}\right)E_4$ & $0$ & $\frac{-f'h+fh'}{h^3}   E_2$  & $-\frac{f'h-h'f}{h^3} E_1$ \\
\hline
$E_3$ &  $\frac{f'h-fh'}{h^3}E_1$  & $\frac{f'h-fh'}{h^3}   E_2$    & $0$& $\frac{f''}{f} E_4$    \\
\hline
$E_4$ &$-\frac{f'h-h'f}{h^3}E_2$ & $\frac{f'h-h'f}{h^3} E_1$  & $-\frac{f''}{f} E_4$  & $0$ \\
\hline
\end{tabular}}
\end{center}
}
\medskip

%
%
%
%
%
%
%
%
%

For instance $R^\nabla_{E_1,E_2}E_1=\frac{2h'fh+2f^2-4h^2}{h^4}E_2,$ etc. To get the corresponding tables of $E_2$ and $E_4$, we can deduce them easily using the fact that $\nabla J=0.$

Now, we compute the first (or the Hermitian) Ricci form $\rho^\nabla$ given by $$\rho^\nabla(X,Y)=-g(R^\nabla(X,Y)E_1,E_2)-g(R^\nabla(X,Y)E_3,E_4).$$
Then
$$\rho^\nabla(E_1,E_2)=\frac{4h^2-2f'h^2-2h'fh}{h^4},$$
$$\rho^\nabla(E_3,E_4)=\frac{h'^2f-h'f'h-h''fh-f''h^2}{h^2f},$$
and
$$\rho^\nabla(E_1,E_3)=\rho^\nabla(E_1,E_4)=\rho^\nabla(E_2,E_3)=\rho^\nabla(E_2,E_4)=0.$$
Hence,
$$\rho^\nabla=\frac{4h^2-2f'h^2-2h'fh}{h^4}\,\sigma_1\wedge\sigma_2+\frac{h'^2f-h'f'h-h''fh-f''h^2}{h^2f}\sigma_3\wedge\sigma_4.$$

Moreover, we can easily check that $d\rho^\nabla=0.$ In particular, 
\begin{equation}\label{closed_ricci}
\left(\frac{4h^2-2f'h^2-2h'fh}{h^2}\right)'=2\,\frac{h'^2f-h'f'h-h''fh-f''h^2}{h^2}.
\end{equation}
Remark that the equation $\rho^\nabla=\lambda\omega$ (where $\lambda$ is a constant) implies that $g$ is K\"ahler. In particular, we can not solve $\rho^\nabla=\lambda\omega$ on the
$m$-Hirzebruch surface. 

\begin{rem}
The second Ricci form is defined as follows
 $$r(X,Y)=-g(R^\nabla(E_1,E_2)X,Y)-g(R^\nabla(E_3,E_4)X,Y).$$
 We obtain that
 $$r=\frac{-h''fh^3-f'h'h^3+h'^2fh^2-2h'f^2h-2f^3+4h^2f}{h^4f}\,\sigma_1\wedge\sigma_2+\left(\frac{2f^2}{h^4}-2\frac{f'}{h^2}-\frac{f''}{f} \right)\sigma_3\wedge\sigma_4.$$
\end{rem}

In~\cite{MR571563}, Gauduchon introduced Einstein-Hermitian metrics which are by definition solutions to the equation $r=u\,\omega$ (for some function $u$). It turns out that, in complex dimension $2$, non-K\"ahler Einstein-Hermitian metrics only exist on a Hopf surface~\cite{MR1477631} (see also~\cite{MR1428883,Angella:aa} for more details).


From the above computations, it follows that:
\begin{thm}\label{chern}
The Chern scalar curvature $s^C$ of $(J,g)$ is given by
\begin{equation}\label{chern_scalar}
s^C=\frac{4hf-2f'fh-2h'f^2+h'^2hf-h'f'h^2-h''fh^2-f''h^3}{h^3f}.
\end{equation}
\end{thm}

\begin{thm}\label{third}
The third scalar curvature $s$ of $(J,g)$ is given by
\begin{equation}\label{third_scalar}
s=\frac{-f''h^4-2f(-h'fh+2f'h^2+f^2-2h^2)}{h^4f}.
\end{equation}
\end{thm}
\begin{rem}
The Riemannian scalar curvature $s^g$ of the metric $g$ is given by (see for example~\cite{MR727843,MR2371700})
$$s^g=-4\frac{h''}{h}-2\frac{f''}{f}-2\frac{h'^2}{h^2}-4\frac{h'f'}{hf}-2\frac{f^2}{h^4}+\frac{8}{h^2}.$$
\end{rem}

\subsection{Construction of Hermitian metrics of constant Chern scalar curvautre}
We would like now to construct constant Chern scalar curvature metrics on the $m$-Hirzebruch surface. 

Suppose that $f=-\frac{1}{2}\phi'\phi$ and $h=\phi$ for some positive decreasing function $\phi$ (so $f=-\frac{1}{2}h'h$).
It follows from Theorem~\ref{chern} that we need to solve
$$-\frac{\phi'''}{\phi'}-\frac{4}{\phi}\phi''+\frac{2}{\phi^2}\phi'^2+\frac{4}{\phi^2}=\lambda,$$
for some constant $\lambda.$

We introduce the function $y(\phi)$ defined by $\phi'=-\sqrt{y(\phi)}$, then $\phi''=\frac{1}{2}y'(\phi)$ and ${\phi'''}=\frac{1}{2}y''(\phi)\,\phi'$
(here $y'(\phi),y''(\phi)$ are derivatives with respect to $\phi$).
Then the equation becomes
$$-\frac{1}{2}y''-\frac{2}{\phi}y'+\frac{2}{\phi^2}y=\lambda-\frac{4}{\phi^2}.$$
The solution is given by
$$y(\phi)=\frac{-\lambda \phi^6+3c_1\phi^5-6\phi^4+3c_2}{3\phi^4},$$
for some constants $c_1$ and $c_2.$

Denote by $\phi(0)=\phi_0>0$ and $\phi(l)=\phi_1>0$. To satisfy the boundary conditions~(\ref{boundary_conditions}) we need then a solution such that $y(\phi_0)=y(\phi_1)=0$ and $y'(\phi_0)=-\frac{4m}{\phi_0}$
and $y'(\phi_1)=\frac{4m}{\phi_1}$ (hence $\phi^{(2p+1)}(0)=\phi^{(2p+1)}(l)=0$, for any $p\geq 0$ because ${\phi'''}=\frac{1}{2}y''(\phi)\,\phi'$ etc).

Then, 

%
%
%
%
\begin{equation}\label{chern_m}
m=\frac{-2\phi_0^4-\phi_0^3\phi_1+\phi_0\phi_1^3+2\phi_1^4}{\phi_0^4-2\phi_0^3\phi_1-3\phi_0^2\phi_1^2-2\phi_0\phi_1^3+\phi_1^4},
\end{equation}

$$\lambda=\frac{-6(3\phi_0^2+4\phi_0\phi_1+3\phi_1^2)}{\phi_0^4-2\phi_0^3\phi_1-3\phi_0^2\phi_1^2-2\phi_0\phi_1^3+\phi_1^4},$$

$$c_1=\frac{-4(\phi_0^3+3\phi_0^2\phi_1+3\phi_0\phi_1^2+\phi_1^3)}{\phi_0^4-2\phi_0^3\phi_1-3\phi_0^2\phi_1^2-2\phi_0\phi_1^3+\phi_1^4},$$

$$c_2=\frac{2\phi_0^4\phi_1^4}{\phi_0^4-2\phi_0^3\phi_1-3\phi_0^2\phi_1^2-2\phi_0\phi_1^3+\phi_1^4},$$

We claim that for any positive integer $m$, there exists $\phi_0>\phi_1>0$ such that~(\ref{chern_m}) is satisfied. Indeed, define $x=\frac{\phi_0}{\phi_1}$,
then Equation~(\ref{chern_m}) is equivalent to $$P(x)=(m+2)x^4-(2m-1)x^3-3mx^2-(2m+1)x+(m-2)=0.$$ Remark that $P(1)=-5m.$ Since $P(x)$
is a polynomial of even degree and $m>0$ there exists a solution of $P(x)=0$ such that $x>1.$
Thus for any positive integer $m$ there exists a solution $\phi_0>\phi_1>0$ of~(\ref{chern_m}). Therefore the function $\phi$ is positive
and decreasing on the interval $(0,l).$ 

Furthermore, we can verify that the function $y(\phi)$ is positive on the interval $(\phi_1,\phi_0)$ with $\phi_0>\phi_1>0$. Indeed,
\begin{equation*}
y(\phi)=\frac{2m(\phi_0-\phi)(\phi_1-\phi)\Phi(\phi)}{\phi^4(\phi_1+\phi_0)(\phi_1-\phi_0)(2\phi_0^2+2\phi_1^2+\phi_0\phi_1)},
\end{equation*}
where $$\Phi(\phi)=\phi^4(3\phi_0^2+4\phi_1\phi_0+3\phi_1^2)+\phi^3(\phi_0^3+\phi_0^2\phi_1+\phi_0\phi_1^2+\phi_1^3)+\phi^2(\phi_0^3\phi_1+\phi_0^2\phi_1^2+\phi_0\phi_1^3)+\phi(\phi_0^3\phi_1^2+\phi_0^2\phi_1^3)+\phi_0^3\phi_1^3.$$
We can then define the function $t(\phi)$ to be
\begin{equation}
t(\phi)=\int_{\phi_0}^\phi\,\frac{d\phi}{-\sqrt{y(\phi)}},
\end{equation}
with $\phi\in(\phi_1,\phi_0)$. We take then the function $\phi(t)$ to be the inverse function of $t(\phi)$ defined on the open interval $(0,a)$
where \[a=\lim_{\phi\to \phi_1} t(\phi).\]
The functions $f$ and $h$ are then solutions to the problem.

We remark that in the case of a solution we have $$\lambda=\frac{-6m(3\phi_0^2+4\phi_0\phi_1+3\phi_1^2)}{-2\phi_0^4-\phi_0^3\phi_1+\phi_0\phi_1^3+2\phi_1^4},$$
and since $\phi_0>\phi_1>0,$ the constant $\lambda$ has to be positive. Moreover, from~(\ref{lee_form}) the Lee form is $$\theta=3\,\frac{h'}{h}dt=3\,d\left(\ln(h)\right).$$
Hence, the metric is conformally K\"ahler. We deduce then Theorem~\ref{chern_constant}.

\subsection{Construction of Hermitian metrics of constant third scalar curvature}

Now, we would like to construct constant third scalar curvature metrics on the $m$-Hirzebruch surface of the form~(\ref{metric}).
Unfortunately, we only succeed in the case $m=1$ i.e. on $\mathbb{CP}^2\# \overline{\mathbb{CP}}^2$.

Suppose that $f=\frac{1}{4}\phi'$ and $h=\sqrt{\phi}$ for some function positive and increasing function $\phi$ (so $f=\frac{1}{2}hh'$).
From Theorem~\ref{third}, it follows that we have to solve
$$-\frac{\phi'''}{\phi'}-\frac{\phi''}{\phi}+\frac{1}{8\,\phi^2}\phi'^2+\frac{4}{\phi}=\lambda,$$
for some constant $\lambda.$

We introduce again the function $y(\phi)$ defined by $\phi'=\sqrt{y(\phi)}$, then $\phi''=\frac{1}{2}y'(\phi)$ and ${\phi'''}=\frac{1}{2}y''(\phi){\phi'}$
(here $y'(\phi),y''(\phi)$ are derivatives with respect to $\phi$).
Then the equation becomes
$$-\frac{1}{2}y''-\frac{1}{2\phi}y'+\frac{1}{8\,\phi^2}y=\lambda-\frac{4}{\phi}.$$
Then the solution is given by

$$y(\phi)=\frac{-8\lambda\phi^{\frac{5}{2}}+15c_1\phi+160\phi^{\frac{3}{2}}+15c_2}{15\sqrt{\phi}},$$
for some constants $c_1$ and $c_2.$ Denote by $\phi(0)=\phi_0>0$ and $\phi(l)=\phi_1>0$. We need then a solution such that $y(\phi_0)=y(\phi_1)=0$ and $y'(\phi_0)=8m$
and $y'(\phi_1)=-8m.$
Then,
\begin{equation}\label{third_m}
m=\frac{-8\phi_0^{\frac{3}{2}}+8\phi_1^{\frac{3}{2}}+6\sqrt{\phi_0}\phi_1-6\sqrt{\phi_1}\phi_0}    {6\phi_0^{\frac{3}{2}}+6\phi_1^{\frac{3}{2}}+9\sqrt{\phi_0}\phi_1+9\sqrt{\phi_1}\phi_0},
\end{equation}

$$\lambda=\frac{40} {2\phi_0+2\phi_1+\sqrt{\phi_1}\sqrt{\phi_0}},$$

$$c_1=\frac{-32\sqrt{\phi_0}\sqrt{\phi_1}(\phi_0+\phi_1+3\sqrt{\phi_1}\sqrt{\phi_0})}     {6\phi_0^{\frac{3}{2}}+6\phi_1^{\frac{3}{2}}+9\sqrt{\phi_0}\phi_1+9\sqrt{\phi_1}\phi_0},$$

$$c_2=\frac{-32\phi_0^{\frac{3}{2}}\phi_1^{\frac{3}{2}}}   {6\phi_0^{\frac{3}{2}}+6\phi_1^{\frac{3}{2}}+9\sqrt{\phi_0}\phi_1+9\sqrt{\phi_1}\phi_0}.$$

Define $x=\frac{\phi_1}{\phi_0}$ then Equation~(\ref{third_m}) becomes 
$$P(x)=(6m-8)x^{\frac{3}{2}}+(9m-6)x+(9m+6)x^{\frac{1}{2}}+(6m+8)=0.$$
Then, $P(1)=30m>0$ and if $m=1$ then we get the existence of a solution of $P(x)=0$ with $x>1$. Explicitly,
If $m=1,$ we get $$\phi_1=\left(\frac{(44+11\sqrt{5})^{\frac{1}{3}})}{2}+\frac{11}{  2(44+11\sqrt{5})^{\frac{1}{3}})  }+\frac{1}{2}\right)^2\phi_0,$$
and then $\lambda\approx \frac{1.11}{\phi_0},c_1\approx -6.52\,\sqrt{\phi_0}$ and $c_2\approx -3.55\,\phi_0^{\frac{3}{2}}.$
So when $m=1$, there exists a solution $\phi_1>\phi_0>0$ of~(\ref{third_m}) making $\phi$ a positive and an increasing function on the interval $(0,l).$ 

Moreover, we can easily verify that the function $y(\phi)$ is also positive
on the interval $(\phi_0,\phi_1).$ We can then define the function $t(\phi)$ to be
\begin{equation}
t(\phi)=\int_{\phi_0}^\phi\,\frac{d\phi}{\sqrt{y(\phi)}},
\end{equation}
with $\phi\in(\phi_0,\phi_1)$. We take then the function $\phi(t)$ to be the inverse function of $t(\phi)$ defined on the open interval $(0,a)$
where \[a=\lim_{\phi\to \phi_1} t(\phi).\]
Then, the functions $f$ and $h$ are solutions to the problem. We remark that the constant $\lambda\approx \frac{1.11}{\phi_0}$ is positive and that the Lee form is $\theta=d(\ln(h))$
and so $(J,g)$ is conformally K\"ahler and so we deduce the following

\begin{thm}\label{third_constant}
There exists on the $1$-Hirzebruch surface conformally K\"ahler metrics of positive constant third scalar scalar curvature of the form~(\ref{metric}).
\end{thm}

\section{Conformally K\"ahler Metrics on ruled surfaces with zero Chern scalar curvature}\label{Sec_4}
In this section we first recall the well-known construction of \textit{admissible K\"ahler metrics} on minimal ruled surfaces~\cite{MR2144249,MR2425136,MR2807093}. Then, by an appropriate conformal rescaling, we will show existence of metrics of zero Chern scalar curvature on some of these surfaces.

\subsection{Admissible K\"ahler Metrics}
Let $\mathbb{M}_m^\mathfrak{g}$ be a ruled surface of the form  $\mathbb{P}(\mathcal{O}_{\Sigma_{\mathfrak g}} \oplus \mathcal L_m  )\longrightarrow \Sigma_{\mathfrak g}$, where $\Sigma_{\mathfrak g}$ is a compact Riemann Surface of genus $\mathfrak{g}$, $\mathcal{L}_m$ is a holomorphic line bundle of degree $m>0$ on $\Sigma_{\mathfrak g}$, and $\mathcal{O}_{\Sigma_{\mathfrak g}}$ is the trivial line bundle. Apostolov--Calderbank--Gauduchon--T\o nnesen-Friedman \cite{MR2425136} called these surfaces \textit{admissible ruled surfaces}, and constructed explicit K\"ahler metrics in every K\"ahler class, called \textit{admissible K\"ahler metrics}\footnote{We note that when the genus of the base $\Sigma_{\mathfrak g}$ is $\mathfrak{g}=0$, we exactly get the Hirzebruch surfaces $\mathbb{M}_m$}. 

The construction can be summarized as follows. Let $g_\Sigma$ be the K\"ahler metric on $\Sigma_{\mathfrak g}$ with constant scalar curvature $2s_\Sigma$ and K\"ahler form $\omega_\Sigma$ such that $c_1(\mathcal L_m)=\left[\frac{\omega_\Sigma}{2\pi}\right]$. It follows that $s_\Sigma = \frac{\chi_\Sigma}{m} = \frac{2-2\mathfrak{g}}{m}$. Note that $s_\Sigma\leq 2$. 

The natural $\mathbb C^*$-action on $\mathcal{L}_m$ extends to a  $\mathbb C^*$-action  on $\mathbb M_m^{\mathfrak g}$. The moment map for the circle action $\z:\mathbb M_m^{\mathfrak g} \rightarrow [-1,1]$ is a smooth function such that the preimages $\mathfrak{z}^{-1}(\pm 1)$ correspond to zero and infinity sections of $\mathbb M_m^{\mathfrak g}$. Let $\mathring{\mathbb M_m^{\mathfrak g}}$ be the complement of these sections. 

Let $x$ be a real parameter in $(0,1)$, and $F(\mathfrak{z}):(-1,1)\rightarrow\mathbb R$  be  a positive smooth function. An \textit{admissible K\"ahler metric} on $\mathring{\mathbb M_m^{\mathfrak g}}$ is then defined by 
\begin{equation}\label{admK}
	g=\frac{1+x\mathfrak{z}}{x} g_\Sigma + \frac{1+x\mathfrak{z}}{F(\z)} d\mathfrak{z}^2 + \frac{F(\z)}{1+x\mathfrak{z}} \theta^2
\end{equation}
where $\theta$ is a connection 1-form for a Hermitian metric with curvature $d\theta =\omega_\Sigma$. The boundary conditions for $g$ to be completed to a smooth metric on the entire $\mathbb M_m^{\mathfrak g}$ are
\begin{equation}\label{adm.K.BC}
	\textnormal{(i) } F(\z)>0,\quad -1<\z<1, \quad \textnormal{(ii) } F(\pm 1)=0, \quad\textnormal{(ii) } F'(\pm 1)=\mp 2(1\pm x).
\end{equation} 
This metric turns out to be K\"ahler with respect to the complex structure $J$ given by the pullback from the base together with the assumption $Jd\z = \frac{F(\z)}{1+x\z} \theta$.

The K\"ahler class can easily be computed as 
$[\omega] = 4\pi E_0 + \frac{2\pi(1-x) m}{x} C$ where $E_0$ and $C$ are the Poincar\'e duals of the zero-section and fiber, respectively \cite{MR2425136}. On the other hand, the K\"ahler cone is described by $\{aE_0 + bC: a,b>0 \}$, and thus, we see that every K\"ahler class (up to scaling) can be represented by an admissible K\"ahler metric as $x$ ranges from $0$ to $1$.

For admissible metrics $g$ we have the following \cite{MR2228318}:
\begin{itemize}
	\item the scalar curvature is 
	\begin{equation}\label{adm.K.scalar}
		s^g = \frac{2s_\Sigma x}{1+x\z } -\frac{F''(\z) }{1+x\z}
	\end{equation}
	\item the Laplacian of a smooth function
	 $p(\z)$ 
	 \begin{equation}\label{adm.K.laplacian}
	 	\Delta^g p = -\frac{[F(\z)p'(z)]'}{(1+x\z)}
	 \end{equation}
\end{itemize}

\subsection{Construction of metrics with zero Chern scalar curvature in the conformal class}
Let $g$ be a K\"ahler metric, and consider the conformally related metric $\tilde g:=e^{2f}g$ for some smooth function $f$. Then, the Chern scalar curvature $\tilde s^C$ of $\tilde{g}$ and the Riemannian scalar curvature $s^g$ of $g$  are related by the identity~\cite{MR742896}
\begin{equation}\label{chern.conf.change}
e^{2f} \tilde s^C = \frac{s^g}{2} + 2 \Delta^g f.
\end{equation}
Now, assume that $g$ is an admissible K\"ahler metric. We are going to consider a special conformal factor  $e^{2f}=\frac{1}{(z+b)^2}$, $|b|>1$, so that $f=-\ln(\z + b)$. This conformal factor is a \textit{holomorphy potential}, and it was used earlier in the context of the study of \textit{conformally K\"ahler Einstein-Maxwell metrics} on admissible ruled surfaces \cite{MR3521556}.

From \eqref{chern.conf.change}, \eqref{adm.K.scalar}, and \eqref{adm.K.laplacian} we compute
\begin{equation}\label{conf.adm.K.chern}
	\tilde s^C = \frac{1}{ 2 (1+x\z)} \left(2s_\Sigma x (\z + b)^2  - (\z+b)^2 F''(\z) + 4 (\z+b) F'(\z) -4F(\z)\right).
\end{equation}
To simplify calculations, we assume that $F(\z)$ is a quartic polynomial. The boundary conditions \eqref{adm.K.BC} (ii)--(iii) imply that $F$ must be of the form
\begin{equation}\label{quartic}
	F(\z) = (1-\z^2) ((1+x\z) -c (1-z^2))
\end{equation}
for some constant $c$. After re-writing \eqref{conf.adm.K.chern} as 
\begin{equation}\label{conf.adm.K.chern.new}
   (\z+b)^2 F''(\z) - 4 (\z+b) F'(\z) +4F(\z) = 2s_\Sigma x (\z + b)^2 - 2 (1+x\z)\tilde s^C,
\end{equation}
and plugging \eqref{quartic} into \eqref{conf.adm.K.chern.new}, we compare the coefficients of powers of $\z^k$, $k=0,1,2,3$, and get the following system.
\begin{align} \label{ls1}
2x - 8bc &= 0\\	\label{ls2}
-12b^2 c - 4c +2 &=2s_\Sigma x \\\label{ls3}
-6b^2 x - 8bc + 4b & = 4s_\Sigma b x - 2\tilde s^C x\\\label{ls4}
4b^2 c - 2b^2 -4bx -4c +4 &=2s_\Sigma b^2 x -2\tilde s^C.
\end{align}
If we regard $b$ as constant, this is a linear system of equations in unknowns $x,c,s_\Sigma,\tilde s^C$, and the system has a unique solution
$$ \tilde s^C=0,\quad  x= \frac{4b}{3b^2 +1},\quad s_\Sigma = -\frac{3b^2 +1}{4b} = -\frac{1}{x},\quad c=\frac{1}{3b^2 +1}.$$ 
We observe that
\begin{itemize}
	\item $0<x<1$ if and only if $b>1$. 
	\item $s_\Sigma=\frac{2-2\mathfrak g}{m}<-1$, and hence $\mathfrak g\geq 2$ and $m\leq 2\mathfrak g-2$.
	\item $F(\z)$ can be factored as 
	\begin{equation}\label{quartic.new}
		F(\z) = \frac{1}{3b^2 +1}(1-\z^2) (\z+3b)(z+b).
	\end{equation} Thus, $b>1$ implies that $F(\z)$ is positive on $(-1,1)$, thereby satisfying condition  (i) of  \eqref{adm.K.BC}.
\end{itemize} 

These observations can be summarized as the following proposition.
\begin{prop}\label{Prop_zero_Chern}	
	Consider the admissible ruled surface $\mathbb{M}_{m}^{\mathfrak g}$ with genus $\mathfrak g\geq 2 $ and degree $m\leq 2\mathfrak g -2$. Let $x: = \frac{m}{2-2 \mathfrak g}$, and let  $b>1$ is the unique number satisfying $x=\frac{4b}{3b^2+1}$. Then, the admissible K\"ahler metric \eqref{admK} corresponding to the parameter $x$, and $F$ given as in \eqref{quartic.new}, is conformal to a metric with zero Chern scalar curvature, via the conformal factor $\frac{1}{(\z+b)^2}$.
\end{prop}

\subsection{Numerical examples }
With the aid of a computer we obtain further examples of constant-Chern-scalar-curvature metrics on ruled surfaces. For a general smooth function $F(\z)$, the ODE in \eqref{conf.adm.K.chern.new} 
\begin{equation*}
(\z+b)^2 F''(\z) - 4 (\z+b) F'(\z) +4F(\z) = 2s_\Sigma x (\z + b)^2 - 2 (1+x\z)\tilde s^C
\end{equation*}
is an Euler equation, which has the general solution 
\begin{equation}\label{genus1F}
	F(\z) = c_1 (\z+b)^4 + c_2 (\z+b) + y_p,
\end{equation}
where a particular solution $y_p$ can be found easily using Maple as 
$$y_p=\frac{2}{3} x\tilde s^C(\z + b)\ln(\z + b) + \frac{1}{18}x\Big((13b + 4\z)\tilde s^C - 6s_\Sigma(b + 3\z)(\z + b)\Big) - \frac{1}{2}\tilde s^C$$

If we impose the boundary conditions \eqref{adm.K.BC} (ii) and (iii) on $F(\z)$, we obtain a nonlinear system of four equations in unknowns $c_1$, $c_2$,  $\tilde s^C$, $s_\Sigma$, $x$, and $b$. Maple can solve the system algebraically in terms of $x$ and $b$. 
{\tiny
\begin{align*}
	\tilde s^C &=\frac{1}{A} \left(-6(b^2-1)(3b^2x-4b+x)\right)\\
	s_\Sigma &= \frac{1}{xA}\left(-3x(b^2-1)(b^3-3b^2x+3b-x) \ln\left(\frac{b-1}{b+1}\right) + (18 b^3 + 6b)x^2 + (-6b^4 - 26b^2 - 4)x + 12b\right) \\
	c_1 & = \frac{1}{A}\left(x((b^3 - b^2x - b + x)\ln(-1 + b) + (-b^3 + b^2x + b - x)\ln(1 + b) + 2b^2 - 3bx + 1)\right) \\
	c_2 &= \frac{1}{3A} \left(6x(b^2-1)(b^3x + 3b^2x + (-x - 4)b + x)\ln(-1 + b) - \ln(1 + b)(b^3x - 3b^2x + (-x + 4)b - x) \right. \\ &\qquad \qquad \left.+ (48b^4 - 20b^2 - 4)x^2 + (-84b^3 + 12b)x + 48b^2 \right) 
\end{align*}}%
where $A=3bx(b^2-1)^2 \ln\left(\frac{b-1}{b+1}\right)+6b^4 x -16b^2x +12b-2x $.
Unfortunately, the solution seems to be too complicated to describe all $b$ and $x$ that lead to a well-defined metric. Nevertheless, fixing the genus $\mathfrak g$, the degree $m$, and $b$, we get some numerical solutions. 

\begin{tabular}{|r|r|r|r|r|r|r|r|}
	\hline 
	$\mathfrak g$ & $m$ & $s_\Sigma$  & $b$ & $x$\hspace{10pt} & $c_1$ \hspace{10pt} & $c_2$ \hspace{10pt} & $\tilde s^C$ \hspace{10pt} \\ \hline
	1 & $>0$ & 0 & 2 &  0.45128 & $-0.04980$ &$-0.13414$ &3.56671\\
	\hline
	2 & 1 & $-2$ & 2 & 0.63961 & $-0.08279$ & 0.81380 & $-1.08112$ \\
	\hline
	2 & 1 & $-2$ & 3 & 0.41604 & $-0.03399$ & $-0.24254$ & 1.48367 \\
	\hline
\end{tabular}

For these three examples, we can also check the positivity condition  (i) of \eqref{adm.K.BC} is satisfied, by graphing $F(\z)$ on Maple. 
\begin{center}
	\includegraphics[scale=0.8]{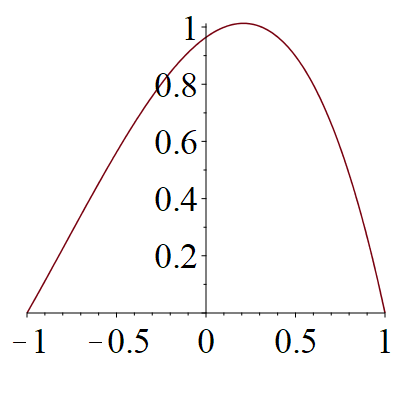} 	\includegraphics[scale=0.8]{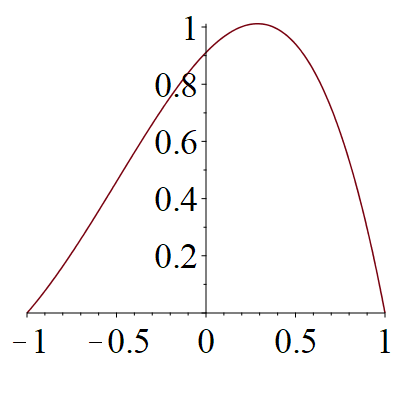}	\includegraphics[scale=0.8]{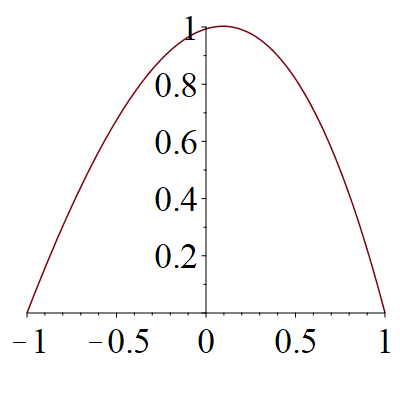}
\end{center}

 We get many such examples by varying the parameters $\mathfrak g$, $m$, and $b$. The metrics are obtained by first considering the admissible K\"ahler metric associated to $F(\z)$ in \eqref{genus1F}, and then conformally rescaling this K\"ahler metric by $\frac{1}{(z+b)^2}$. We would like to remark that the sign of the constant Chern scalar curvature metrics we obtained on ruled surfaces over $\Sigma_{\mathfrak g}$ of genus $\mathfrak g  \geq 2$ can be positive, negative, or zero. Also, when $\mathfrak g =1$, $x$ is always in $(0,1)$ whenever $b\geq 2$.
 
 We hope that these numerical examples can lead to an interesting theory of such metrics on higher genus ruled surfaces.
 
 \section{Gauduchon critical metrics}~\label{Sec_5}
On a closed complex $(M,J)$ of real dimension $2n$, Gauduchon studied in~\cite{MR567760,MR742896} the critical metrics of the total Chern scalar curvature and the third scalar curvature
$$\mathcal{C}(g)=\displaystyle\int_Ms^Cv_g,\quad\mathcal{S}(g)=\displaystyle\int_Msv_g,$$
in the space of all Hermitian metrics and also in a conformal class. 

When restricted to a conformal class with fixed total volume, the critical
points of these functionals are Hermitian metrics $g$ such that respectively 
$$s^C+\frac{n}{2(n-1)}\delta^g\theta,\quad s+\frac{1}{2(n-1)}\delta^g\theta$$ are constant~\cite{MR742896}. We remark that these functionals can be extended to the almost-Hermitian setting
and then we will have the same Euler-Lagrange equations because the conformal changes of the Chern scalar curvature and the third scalar curvature are the same as in the integrable case~\cite[Corollary 4.5]{Lejmi:2017aa}

First, we would like to construct metrics of the form~(\ref{metric}) with $s^C+\frac{n}{2(n-1)}\delta^g\theta=s^C+\delta^g\theta$ being constant.
Let $(J,g)$ be the Hermitian strcuture given by~(\ref{complex_structure}) and~(\ref{metric}). Then, we have 
\begin{equation}\label{gauduchon_critical}
s^C+\delta^g\theta=\frac{-3h''fh^2-f''h^3-h'^2fh-3h'f'h^2-2h'f^2+2f'fh+4fh}{h^3f}.
\end{equation}
\begin{thm}
There exists on the $1$-Hirzebruch surface conformal K\"ahler metrics of the form~(\ref{metric}) with $s^C+\delta^g\theta$ being a positive constant.
\end{thm}

\begin{proof}
Suppose that $f=-\frac{7}{5}\phi'\phi$ and $h=\phi$ for some positive and decreasing function $\phi$ (so $f=-\frac{7}{5}h'h$).
From~(\ref{gauduchon_critical}), we get the equation
$$-\frac{\phi'''}{\phi'}-\frac{59}{5\phi}\phi''-\frac{4}{\phi^2}\phi'^2+\frac{4}{\phi^2}=\lambda,$$
where $\lambda$ is a constant.

We introduce again the function $y(\phi)$ defined by $\phi'=-\sqrt{y(\phi)}$, then $\phi''=\frac{1}{2}y'(\phi)$ and ${\phi'''}=\frac{1}{2}y''(\phi){\phi'}$
(here $y'(\phi),y''(\phi)$ are derivatives with respect to $\phi$).
Then the equation becomes
$$-\frac{1}{2}y''-\frac{59}{10\phi}y'-\frac{4}{\phi^2}y=\lambda-\frac{4}{\phi^2}.$$
The solution is given by
\begin{equation}\label{sol}
y(\phi)=-\frac{5\lambda \phi^2}{84}+c_1\phi^{-\frac{4}{5}}+c_2\phi^{-10}+1,
\end{equation}
for some constants $c_1$ and $c_2.$

Denote by $\phi(0)=\phi_0>0$ and $\phi(l)=\phi_1>0$. We need a solution such that $y(\phi_0)=y(\phi_1)=0$ and $y'(\phi_0)=-\frac{10m}{7\phi_0}$
and $y'(\phi_1)=\frac{10m}{7\phi_1}.$

When $m=1$, then using Maple we get that $$\phi_1\approx 0.155\, \phi_0.$$ The constant $\lambda\approx\frac{13.371}{\phi_0^2}$ is then positive.
Hence, we obtain a positive
and decreasing function $\phi$ on the interval $(0,l)$. We can also check using Maple that the function $y(\phi)$ is positive on the interval $(\phi_1,\phi_0)$. 
The rest of the argument is the same as in the construction of metrics of constant Chern scalar curvature. 
\end{proof}
\subsection{The existence of Gauduchon critical metrics in a conformal class}

Let $(M,J,g)$ be a closed almost-Hermitian manifold of real dimension $2n$, 
Suppose that $\tilde{g}=e^{2f}g$ then we have the following conformal changes~\cite{MR742896,Lejmi:2017aa} :
$$e^{2f}\left(\tilde{s}^C+\frac{n}{2(n-1)}{\delta}^{\tilde{g}}\tilde{\theta} \right)=\left(s^C+\frac{n}{2(n-1)}\delta^g\theta\right)+2n\Delta^gf-n(2n-2)g(df,df),$$
$$e^{2f}\left(\tilde{s}+\frac{1}{2(n-1)}{\delta}^{\tilde{g}}\tilde{\theta} \right)=\left(s+\frac{1}{2(n-1)}\delta^g\theta\right)+2\Delta^gf-(2n-2)g(df,df),$$
where $\Delta^g$ is the Riemannian Laplacian with respect to the metric $g.$ Consider the conformal change $e^{2f}=\phi^{\frac{2}{n-1}}$ then the equations become:
\begin{equation}\label{yamabe_1}
\lambda \phi^{\frac{n+1}{n-1}}=\left(s^C+\frac{n}{2(n-1)}\delta^g\theta\right)\phi+\frac{2n}{n-1}\Delta^g\phi,
\end{equation}
\begin{equation}\label{yamabe_2}
\lambda \phi^{\frac{n+1}{n-1}}=\left(s+\frac{1}{2(n-1)}\delta^g\theta\right)\phi+\frac{2}{n-1}\Delta^g\phi.
\end{equation}
for some constant $\lambda$. As observed by Gauduchon~\cite{MR742896}, these are of Yamabe-type equation~\cite{MR0125546} with the critical exponent. 

Now, we concentrate on the equation~(\ref{yamabe_1}) (the following arguments can be also applied to the equation~(\ref{yamabe_2})). We consider the functional~\cite{MR0125546}
\begin{equation*}
E(\phi)=\frac{\displaystyle\int_M \frac{2n}{n-1}\|\nabla^g \phi\|_g^2+\left( s^C+\frac{n}{2(n-1)}\delta^g\theta   \right)\,\phi^2   v_g}{\left(\displaystyle\int_M\phi^{\frac{2n}{n-1}}v_g\right)^{\frac{n-1}{n}}}.
\end{equation*}
Then,
\begin{equation*}
\frac{d}{dt}\left( E(\phi+t\psi) \right)|_{t=0}=\frac{2\left( \displaystyle\int_M   \left(\frac{2n}{n-1} \Delta^g\phi+\left( s^C+\frac{n}{2(n-1)}\delta^g\theta   \right)-E(\phi)\phi^{\frac{n+1}{n-1}}  \left(\displaystyle\int_M\phi^{\frac{2n}{n-1}}v_g\right)^{-1}   \right)\psi v_g   \right)}{\left(\int_M\phi^{\frac{2n}{n-1}}v_g\right)^{\frac{n-1}{n}}}
\end{equation*}
Therefore $\phi$ is a critical point of the functional $E$ if and only if it satisfies the equation~(\ref{yamabe_1}) with $\lambda= E(\phi) \left(\displaystyle\int_M\phi^{\frac{2n}{n-1}}v_g\right)^{-1}.$

On the other hand, $$E(\phi)=E(\tilde{g})=\frac{\displaystyle\int_M \tilde{s}^C+\frac{n}{2(n-1)}{\delta}^{\tilde{g}}\tilde{\theta}\, v_{\tilde{g}}}{\left(\displaystyle\int_Mv_{\tilde{g}}\right)^{\frac{n-1}{n}}}=\frac{\displaystyle\int_M \tilde{s}^C\, v_{\tilde{g}}}{\left(\displaystyle\int_Mv_{\tilde{g}}\right)^{\frac{n-1}{n}}},$$
where $\tilde{g}=\phi^{\frac{2}{n-1}}g.$ Set 
\begin{eqnarray*}
G(J,[g])&=&\inf\{E(\tilde{g})\,|\,\tilde{g}\text{ conformal to } g\}   \\
&=&\inf\{E(\phi)\,|\,\phi \text{ smooth and positive function on } M\}.
\end{eqnarray*}
Then $G(J,[g])$ is a conformal invariant. As a straightforward application of Aubin's work~\cite{MR0431287}, we get
\begin{prop}\label{Aubin_result}\cite{MR0431287}
Let $(M,J,g)$ be a closed almost-Hermitian of real dimension $2n$. If $G(J,[g])<2n^2\vol_{S^{2n}}^{\frac{1}{n}}$, where $\vol_{S^{2n}}$ is the volume of the unit sphere in $\mathbb{R}^{2n+1}$, then there exists a metric $\tilde{g}\in[g]$ such that $\tilde{s}^C+\frac{n}{2(n-1)}{\delta}^{\tilde{g}}\tilde{\theta}=G(J,[g]).$ Moreover, if $G(J,[g])\leq 0,$ then there exists a unique solution (up to a constant) to the equation~(\ref{yamabe_1}).
\end{prop}
\begin{prop}~\cite{MR0431287}
Let $(M,J,g)$ be a closed almost-Hermitian of real dimension $2n$. Suppose that there exists a point $p\in M$ such that $$s^C(p)+\frac{n}{2(n-1)}\left(\delta^g\theta\right)(p)-\frac{n}{2n-1}s^g(p)<0,$$
there there exists a metric $\tilde{g}\in[g]$ such that $\tilde{s}^C+\frac{n}{2(n-1)}{\delta}^{\tilde{g}}\tilde{\theta}=G(J,[g]).$
\end{prop}
\begin{defn}\label{fundamental_constant}
The fundamental constant $C(J,[g])$ (see~\cite{MR0486672,MR779217,MR1712115}) of a closed almost-Hermitian manifold $(M,J,g)$ is defined as 
$$C(J,[g])=\displaystyle\int_M s_0^C v_{g_0},$$
where $g_0$ is the unique Gauduchon metric in $[g]$ with total volume equal to $1$ and $s_0^C$ the Chern scalar curvature of $(J,g_0).$
\end{defn}
\begin{cor}
Suppose that the fundamental constant $C(J,[g])\leq 0,$ then there exists a unique $\tilde{g}\in[g]$ (up to a constant) such that  $\tilde{s}^C+\frac{n}{2(n-1)}{\delta}^{\tilde{g}}\tilde{\theta}=G(J,[g]).$
\end{cor}
\begin{proof}
Recall that $G(J,[g])=\inf\{E(\tilde{g})\,|\,\tilde{g}\text{ conformal to } g\}$ where $E(\tilde{g})=\frac{\displaystyle\int_M \tilde{s}^C\, v_{\tilde{g}}}{\left(\displaystyle\int_Mv_{\tilde{g}}\right)^{\frac{n-1}{n}}}.$
Hence, if $C(J,[g])\leq 0$ then $G(J,[g])\leq 0$ and then we apply Proposition~\ref{Aubin_result}.
\end{proof}
\begin{rem}
We know from~\cite{MR3696598} that if~$C(J,[g])\leq 0$, then there exists a metric $\tilde{g}\in[g]$ such that $\tilde{s}^C=C(J,[g]).$ However, the metric $\tilde{g}$
is not necessarily Gauduchon hence $\tilde{s}^C+\frac{n}{2(n-1)}{\delta}^{\tilde{g}}\tilde{\theta}$ of the metric $\tilde{g}$ is not necessarily constant. 
\end{rem}

When $G(J,[g])>0,$ we can also deduce the following from Bahri and Brezis's work~\cite{MR967364} (see also~\cite{MR992212})
\begin{prop}\cite{MR967364}
Let $(M,J,g)$ be a closed almost-Hermitian of real dimension $2n$ and suppose that $G(J,[g])>0.$ If $n=2$ or $n=3$, then there exists a metric $\tilde{g}\in[g]$ such that $\tilde{s}^C+\frac{n}{2(n-1)}{\delta}^{\tilde{g}}\tilde{\theta}=G(J,[g]).$
If $n\geq 4$ and if the homology group $H_1(M,\mathbb{Z}_2)$ or $H_2(M,\mathbb{Z}_2)$ is non trivial then there exists a solution to the equation~(\ref{yamabe_1}).
\end{prop}
\begin{proof}
To apply Bahri and Brezis's result~\cite{MR967364}, we only need to prove that there exists a metric $\tilde{g}\in[g]$ such that $\tilde{s}^C+\frac{n}{2(n-1)}{\delta}^{\tilde{g}}\tilde{\theta}$ is everywhere positive and thus
the operator $\Delta^{\tilde{g}}+\left(\tilde{s}^C+\frac{n}{2(n-1)}{\delta}^{\tilde{g}}\tilde{\theta}\right)$ will be strictly positive.
The existence of such metric $\tilde{g}\in[g]$ follows from Yamabe's work in the positive case (see~\cite{MR0279731} and also~\cite[Theorem p.150]{MR1636569}
\end{proof}

\bibliographystyle{abbrv}

\bibliography{biblio_Hirzebruch}

\end{document}